\documentclass[11pt, oneside]{amsart}  	
\usepackage{geometry}               
\usepackage{graphicx}				
\usepackage{amsthm,amsmath,amssymb,graphicx,mathrsfs,braket,latexsym}

\newtheorem{thm}{Theorem}[section]
\newtheorem{prop}[thm]{Proposition}
\newtheorem{lem}[thm]{Lemma}
\newtheorem{cor}[thm]{Corollary}
\theoremstyle{definition}
\newtheorem{defi}[thm]{Definition}
\newtheorem{rem}[thm]{Remark}
\newtheorem{ex}[thm]{Example}

\newtheorem{nota}[thm]{Notation}

\newtheorem{setting}{Setting}

\DeclareMathOperator{\supp}{supp}

\DeclareMathOperator{\Patch}{Patch}

\DeclareMathOperator{\LF}{LF}
\DeclareMathOperator{\Map}{Map}
\DeclareMathOperator{\UD}{UD}

\DeclareMathOperator{\Del}{Del}
\DeclareMathOperator{\calC}{Cl}
\DeclareMathOperator{\Homeo}{Homeo}
\DeclareMathOperator{\Cpt}{Cpt}
\DeclareMathOperator{\WDC}{WDC}

\newcommand{\calT}{\mathcal{T}}
\newcommand{\calA}{\mathcal{A}}

\newcommand{\calP}{\mathcal{P}}
\newcommand{\calQ}{\mathcal{Q}}
\newcommand{\calR}{\mathcal{R}}

\newcommand{\calU}{\mathcal{U}}

\newcommand{\e}{\varepsilon}
\newcommand{\Rd}{\mathbb{R}^d}

\newcommand{\sci}{\wedge}

\title[The common structure]{The common structure for objects in aperiodic order and the
theory of local matching topology}
\author{Yasushi Nagai}
\address{Montanuniversit\"at, Department Mathematik und Informationstechnologie,
Lehrstuhl f\"ur Mathematik und Statistik,
Franz Josef Strasse 18, A-8700 Leoben, Austria}
\email{yasushi.nagai@unileoben.ac.at}
\date{\today}							
\thanks{The author was supported by the project I3346 of the Japan Society for the Promotion of Science (JSPS) and the Austrian Science Fund (FWF)}

\begin{document}
\maketitle

\begin{abstract}
       In aperiodic order, non-periodic but ``ordered''
     objects such as tilings, Delone sets, functions and measures are investigated.
         In this article we depict the common structure of these objects by using
       the general framework of abstract pattern spaces. In particular, using the
      common structure
       we define local matching topology and uniform structure
       for objects such as tilings 
      in quite a general space and a symmetry group. We prove Hausdorff property of the
      topology and the completeness of the uniform structure under a mild assumption.
      We also prove finite local complexity implies the compactness of the continuous
 hull and often the converse holds.
\end{abstract}

\section{Introduction}
Ever since quasicrystals were discovered, mathematical objects such as tilings,
Delone (multi) sets, weighted Dirac combs and almost periodic functions have been
investigated, especially on their diffraction nature and the
connection with topology and the
theory of dynamical systems. In this context the continuous hulls and the corresponding
dynamical systems are important, where the former are
the closures of their orbits and
the latter are obtained by group actions on continuous hulls and
geometric analogues for symbolic dynamics.
The choice of topology is crucial here, since we want the continuous hulls to be compact.
The simplest topology (and uniform structure) is the local matching topology (and
local matching
uniform structure). If the ambient space where
 the above objects live is $\Rd$, it is known
that a condition called finite local complexity (FLC) assures that the continuous hulls
are compact with respect to the local matching topology. In the proof for this claim,
the fact that the local matching uniform structure is complete is tacitly used.
For discrete subsets (which include Delone sets), the completeness of
the local matching uniform structure is proved in \cite{MR1798991} and \cite{MR2137108}.
(These papers deal with the cases where the ambient space is not necessarily
Euclidean, but mathematically there
is no necessity to restrict ourselves to the Euclidean case.)
However there seems to be no analogous results for tilings, Delone multi sets,
weighted Dirac combs and functions, if the ambient space in which these objects live
is general. Since there are results on construction of non-periodic tilings in general
spaces
 (\cite{MR1145337},\cite{MR1452434},\cite{MR1658579}),
 it is worthwhile to prove such completeness. In this article we prove such
completeness in full generality.

The argument is based on a general framework to discuss those objects in a unified manner.
The essential structure of those objects is cutting-off operation and group action of
sliding objects.
For example, if $D$ and $C$ are subsets of a set $X$, we
can ``cut off'' $D$ by $C$ by considering $D\cap C$. If a group $\Gamma$ acts on $X$,
we can ``slide'' $D$ by considering $\gamma D$, for each $\gamma\in\Gamma$.
For each object, these structures have common properties and we can axiomatize them.
A set with a cutting-off operation and a group action that satisfy those axioms
are called \emph{abstract pattern spaces} and the elements of abstract pattern spaces
are called \emph{abstract patterns}. Naturally the above objects such as tilings and
Delone sets are abstract patterns.

The structure of cutting-off operation and group action with the axioms are common
structure of objects such as tilings and Delone sets.
The common structure is enough to define the local matching
topology and uniform structure on abstract pattern spaces.
Under a mild assumption, the topology is Hausdorff 
(Proposition \ref{prop_sufficint_condition_hausdorff}) and metrizable
(Corollary \ref{102557_9Aug18}).
We also prove that often on a subspace of abstract pattern space, the local matching
uniform structure is complete (Theorem \ref{thm_sigma_complete_general_ver}).
This means that FLC of an abstract pattern implies the compactness of the continuous hull
(Theorem \ref{thm_compact_continuous_hull}).

In Section 2 we give an introduction of the theory of abstract pattern spaces, although
we omit most of the proofs, which can be found in \cite{Nagai3rd}.
In Section 3 we define the local matching topology and uniform structure and investigate
the properties, such as completeness and metrizability.

\begin{setting}
     Here is the setting of this article. The symbol $X$ represents a proper metric space.
      The metric on $X$ is denoted by $\rho_X$.
     $\Gamma$ is a group and $\rho_{\Gamma}$ is a left-invariant proper metric
       on $\Gamma$.
       We assume $\Gamma$ acts on $X$ as isometries and the action is jointly continuous,
 that is, the map $\Gamma\times X\ni(\gamma,x)\mapsto\gamma x\in X$ is continuous,
 where the domain is endowed with the product topology.
      We take $x_0\in X$ and use it as a reference point throughout the article.
\end{setting}

\begin{nota}
      For $x\in X$ and $r>0$, we define the closed ball $B(x,r)$ via
      \begin{align*}
            B(x,r)=\{y\in X\mid \rho_X(x,y)\leqq r\}.
      \end{align*}
       Similarly for $\gamma\in\Gamma$ and $r>0$ we set
       \begin{align*}
	    B(\gamma,r)=\{\eta\in\Gamma\mid\rho_{\Gamma}(\gamma,\eta)\leqq r\}.
       \end{align*}

      The one-dimensional torus is denoted by $\mathbb{T}$:
       \begin{align*}
	     \mathbb{T}=\{z\in\mathbb{C}\mid |z|=1\}.
       \end{align*}
\end{nota}

\section{General theory of abstract pattern spaces}
\label{section_abstract pattern_space}

In this section we summarize the contents of \cite{Nagai3rd} to introduce abstract
pattern spaces discussed in Introduction.

Objects such as tilings (Example \ref{example_patch}) and Delone sets
 (Example \ref{example_UD}) admit the following
structures, which play important roles explicitly or implicitly.
\begin{enumerate}
 \item They admit cutting-off operation. For example, if $\calT$ is a tiling in $\Rd$ and
       $C\subset\Rd$, we can ``cut off'' $\calT$ by $C$ by considering
       \begin{align*}
	    \calT\sci C=\{T\in\calT\mid T\subset C\}.
       \end{align*}
       By this operation we forget the behavior of $\calT$ outside $C$.
\item Some of the objects ``include'' other objects. For patches this means the usual
      inclusion of two sets; for measures this means one measure is a restriction of
      another.
 \item They admit gluing operation. For example, suppose
       $\{\calP_i\mid i\in I\}$ is a family of patch
       such that if $i,j\in I, T\in\calP_i$ and $S\in\calP_j$, then either
       $S=T$ or $S\cap T=\emptyset$.
       Then we can ``glue'' $\calP_i$'s and obtain a patch $\bigcup_{i\in I}\calP_i$.
 \item There are ``zero elements'', which contains nothing.
       For example, empty set is a patch that contains no tiles; zero function also
       contains no information.
         Such a zero element is often unique for each category
       of objects. 
\end{enumerate}

We will axiomatize the cutting-off operation should satisfy and a set
with a cutting-off operation that obey the axiom is called
abstract pattern space.The other structures in the list are captured by the
cutting-off operation.

First in Subsection \ref{subsection_def_pat_sp} we give the axiom and define
abstract pattern spaces. In Subsection \ref{subsection_order_pat_sp} we deal with
the ``inclusion'' in the list. In Subsection \ref{subsection_glueable_pat_sp} we discuss
``gluing'' operation in the list, by describing abstract pattern spaces in which we
``often'' glue abstract patterns. Finally in Subsection \ref{subsection_zero_elements}
we discuss ``zero elements'' in the list and give a sufficient conditions for it to be
unique. In Subsection \ref{subsection_gammma-abstract pattern_space} we deal with
abstract pattern spaces with $\Gamma$-action, which are called $\Gamma$-abstract
pattern spaces.

\subsection{Definition and examples of abstract pattern space}
\label{subsection_def_pat_sp}

\begin{defi}
       The set of all closed subsets of $X$ is denoted by $\calC(X)$.
\end{defi}

\begin{defi}
      A non-empty set $\Pi$ equipped with a map
      \begin{align}
           \Pi\times\calC(X)\ni (\mathcal{P},C)\mapsto \calP\sci C\in\Pi \label{scissors_operation}
      \end{align}
      such that
      \begin{enumerate}
       \item $(\calP\sci C_1)\sci C_2=\calP\sci (C_1\cap C_2)$ for any $\calP\in\Pi$ and 
                any $C_1,C_2\in\calC(X)$, and
       \item for any $\calP\in\Pi$ there exists $C_{\calP}\in\calC(X)$ such that
             \begin{align*}
	         \calP\sci C=\calP\iff C\supset C_{\calP},
	     \end{align*}
               for any $C\in\calC(X)$,
      \end{enumerate}
      is called an \emph{abstract pattern space} over $X$.
      The map (\ref{scissors_operation}) is called the \emph{cutting-off operation} of
      the abstract pattern space $\Pi$.
      The closed set $C_{\calP}$ that appears in 2.\  is unique.
      It is called the \emph{support} of $\calP$ and is represented by $\supp\calP$.
      Elements in $\Pi$ are called \emph{abstract patterns} in $\Pi$.
\end{defi}

The following lemma describes a relation between the support and the
cutting-off operation.
\begin{lem}\label{lemma_support_calP_sci_C}
    Let $\Pi$ be an abstract pattern space over $X$.
    For any $\calP\in\Pi$ and $C\in\calC (X)$, we have $\supp(\calP\sci C)\subset(\supp\calP)\cap C$.
\end{lem}
\begin{proof}
\begin{align*}
   (\calP\sci C)\sci ((\supp\calP)\cap C)=(\calP\sci\supp\calP)\sci C=\calP\sci C.
\end{align*}
\end{proof}

We now give several examples of abstract pattern spaces.

\begin{ex}[The space of patches in $X$]\label{example_patch}
      An open, nonempty and bounded subset of $X$ is called a \emph{tile} (in $X$).
     A set $\calP$ of tiles such that if $S,T\in\calP$, then either $S=T$ or $S\cap T=\emptyset$
     is called a \emph{patch} (in $X$).
     The set of all patches in $X$ is denoted by $\Patch(X)$.
     For $\calP\in\Patch(X)$ and $C\in\calC(X)$, set
     \begin{align*}
           \calP\sci C=\{T\in\calP\mid T\subset C\}.
     \end{align*}
     With this cutting-off
     operation $\Patch(X)$ becomes an abstract pattern space over $X$.
     For $\calP\in\Patch(X)$, its support is
     \begin{align*}
           \supp\calP=\overline{\bigcup_{T\in\calP}T}.
     \end{align*}  
       Patches $\calP$ with $\supp\calP=X$ are called \emph{tilings}.
\end{ex}

\begin{rem}
     Usually tiles are defined to be (1) a compact set that is the closure of its
 interior \cite{BH},
     or in Euclidean case,    (2) a polygonal subset of $\Rd$ \cite{Wh} or
 (3) a homeomorphic image of closed unit
 ball (for example, \cite{AP}).
     The advantage of our definition is that we can give punctures to tiles
    and we do not need to consider labels (Example \ref{ex_L-labeled_tiling}).

 The usual labeled tilings (Example \ref{ex_L-labeled_tiling})
      are often MLD with tilings with open tiles (Example \ref{example_patch}).
     It is the cutting-off operation that is essential and the definition of tiles
    is not essential.
\end{rem}

\begin{ex}[The space of labeled patches, \cite{MR1976605}, \cite{MR2851885}]\label{ex_L-labeled_tiling}
      Let $L$ be a set. An \emph{$L$-labeled tile}
        is a pair $(T,l)$ of a compact subset $T$
       of $X$ and $l\in L$, such that $T=\overline{T^{\circ}}$ (the closure of the
       interior).
       An \emph{$L$-labeled patch} is a collection $\calP$ of $L$-labeled tiles such that
        if $(T,l), (S,k)\in\calP$, then either $T^{\circ}\cap S^{\circ}=\emptyset$, or
	$S=T$ and $l=k$. For an $L$-labeled patch $\calP$, define the support of $\calP$
       via
       \begin{align*}
	    \supp \calP=\overline{\bigcup_{(T,l)\in\calP}T}.
       \end{align*}
       An $L$-labeled patch $\calT$ with $\supp\calT=X$ is called an \emph{$L$-labeled tiling}.
      Sometimes we suppress $L$ and call such tilings labeled tilings.

      For an $L$-labeled patch $\calP$ and $C\in\calC(X)$, define a cutting-off
       operation via
       \begin{align*}
	      \calP\sci C=\{(T,l)\in\calP\mid T\subset C\}.
       \end{align*}
       The space $\Patch_L(X)$ of all $L$-labeled patches is a pattern space over $X$
        with this cutting-off operation. 
\end{ex}

\begin{ex}[The space of all locally finite subsets of $X$]\label{example_LF(X)}
     Let $\LF(X)$ be the set of all locally finite subsets of $X$; that is,
     \begin{align*}
           \LF(X)=\{D\subset X\mid \text{for all $x\in X$ and $r>0$, $D\cap B(x,r)$ is finite}\}.
     \end{align*}
     With the
     usual intersection
     $\LF(X)\times\calC(X)\ni(D,C)\mapsto D\cap C\in\LF(X)$ of two subsets
     of $X$, $\LF(X)$ is an abstract pattern space over $X$.
     For any $D\in\LF(X)$, its support is $D$ itself.     
\end{ex}

\begin{ex}[The space of all uniformly discrete subsets]\label{example_UD}
      We say, for $r>0$,
     a subset $D$ of $X$ is \emph{$r$-uniformly discrete} if
     $\rho_X(x,y)>r$ for any $x,y\in D$ with $x\neq y$.
     The set $\UD_r(X)$ of all $r$-uniformly discrete subsets of $X$ is an abstract pattern space over $X$
     by the usual intersection as a cutting-off operation.
     If $D$ is $r$-uniformly discrete for some $r>0$, we say $D$ is \emph{uniformly discrete}.
     The set $\UD(X)=\bigcup_{r>0}\UD_r(X)$ of all uniformly discrete subsets of $X$ is
     also an abstract pattern space over $X$.

       Take a positive real number $R$.
      A subset $D$ of $X$ is \emph{$R$-relatively dense} if whenever we take $x\in X$, the
      intersection $D\cap B(x,R)$ is non-empty. A subset of $X$ is \emph{relatively dense}
       if it is $R$-relatively dense for some $R>0$. The uniformly discrete and
 relatively dense subsets of $X$ are called \emph{Delone sets}.
\end{ex}

\begin{ex}\label{example_2^X_calC(X)}
     With the usual intersection of two subsets of $X$ as a cutting-off operation,
     the set $2^X$ of all subsets of $X$ and $\calC(X)$ are abstract pattern spaces over $X$. For example, the union of all Ammann bars for a Penrose tilings is an abstract
 pattern.
\end{ex}

\begin{ex}[The space of maps]\label{example_map}
      Let $Y$ be a nonempty set.
      Take one element $y_0\in Y$ and fix it.
      The abstract pattern space $\Map(X,Y,y_0)$ is defined as follows:
      as a set the space is equal to $\Map(X,Y)$ of all mappings from $X$ to $Y$;
      for $f\in\Map(X,Y,y_0)$ and $C\in\calC(X)$, the cutting-off operation is defined by
      \begin{align*}
           (f\sci C)(x)=
           \begin{cases}
	         f(x)&\text{if $x\in C$}\\
                 y_0&\text{if $x\notin C$}.
	   \end{cases}
      \end{align*}
      With this operation $\Map(X,Y,y_0)$ is an abstract pattern space over $X$
      and for $f\in\Map(X,Y,y_0)$ its support is 
      $\supp f=\overline{\{x\in X\mid f(x)\neq y_0\}}$.
\end{ex}

\begin{ex}[The space of measures]\label{ex_space_of_measures}
     Let $C_c(X)$ be the space of all continuous and complex-valued functions on $X$
     which have compact supports.
     Its dual space $C_c(X)'$ with respect to the inductive limit topology
     consists of Radon charges, that is,
     the maps $\Phi
     \colon C_c(X)\rightarrow\mathbb{C}$ such that there is a unique positive
     Borel measure $m$ and a Borel measurable map $u\colon X\rightarrow\mathbb{T}$
     such that
     \begin{align*}
      \Phi(\varphi)=\int_{X}\varphi udm
     \end{align*}
     for all $\varphi\in C_c(X)$.
     For such $\Phi$ and $C\in\calC (X)$ set
     \begin{align*}
         (\Phi\sci C)(\varphi)=\int_C \varphi udm
     \end{align*}     
     for each $\varphi\in C_c(X)$.
     Then the new functional $\Phi\sci C$ is a Radon charge.
     With this operation $C_c(X)'\times\calC (X)\ni(\Phi,C)\mapsto
      \Phi\sci C\in C_c(X)'$,
     the space $C_c(X)'$ becomes an abstract pattern space over $X$.
\end{ex}

Next we investigate abstract pattern subspaces.
The relation between an abstract pattern space and its abstract
pattern subspaces is similar to the one
between  a set with a group action and its invariant subsets.

\begin{defi}\label{def_subspace}
     Let $\Pi$ be an abstract pattern space over $X$.
     Suppose a non-empty subset $\Pi'$ of $\Pi$ satisfies the condition
     \begin{align*}
          \text{$\calP\in\Pi'$ and $C\in\calC(X)\Rightarrow\calP\sci C\in\Pi'$}.
     \end{align*}
     Then $\Pi'$ is called an \emph{abstract pattern subspace} of $\calP$.
\end{defi}

\begin{rem}
      If $\Pi'$ is an abstract
      pattern subspace of an abstract pattern space $\Pi$, then $\Pi'$ is a 
      abstract pattern space by restricting the cutting-off operation.
\end{rem}

\begin{ex}
      $\calC(X)$ is an abstract pattern subspace of
      $2^X$.
       $\LF(X)$ is an abstract
 pattern subspace of $\calC(X)$ and
           $\UD_r(X)$ is an abstract pattern subspace of $\UD(X)$ 
      for each $r>0$. 
       Since we assume the metrics we consider are proper,
      $\UD(X)$ is an abstract pattern subspace of
      $\LF(X)$.
\end{ex}

Next we deal with a way to construct new abstract pattern space from old ones, namely,
taking product.

\begin{lem}\label{lemmma_product_abstract pattern_space}
      Let $\Lambda$ be an index set and $\Pi_{\lambda},\lambda\in\Lambda$, is a 
     family of abstract pattern spaces over $X$.
     The direct product $\prod_{\lambda}\Pi_{\lambda}$ becomes an abstract pattern space over $X$ by 
     a cutting-off operation
     \begin{align*}
      (\calP_{\lambda})_{\lambda\in\Lambda}\sci 
        C=(\calP_{\lambda}\sci C)_{\lambda\in\Lambda}.
     \end{align*}
     for $(\calP_{\lambda})_{\lambda}\in\prod_{\lambda}\Pi_{\lambda}$ and $C\in\calC(X)$.
     The support is given by
     $\supp(\calP_{\lambda})_{\lambda}=\overline{\bigcup_{\lambda}\supp\calP_{\lambda}}$.
\end{lem}

\begin{defi}\label{def_product_abstract pattern_sp}
    Under the same condition as in Lemma \ref{lemmma_product_abstract pattern_space}, we call
    $\prod\Pi_{\lambda}$ the product abstract pattern space of $(\Pi_{\lambda})_{\lambda}$.
\end{defi}

By taking product, we can construct the space of uniformly discrete multi sets,
as follows.

\begin{ex}[uniformly discrete multi set, \cite{MR1976605}]\label{ex_multi_set}
     Let $I$ be a set and $r>0$. Consider a space $\UD^I_r(X)$, defined via
     \begin{align*}
          \UD^I_r(X)=\{(D_i)_{i\in I}\in\prod_{i\in I}\UD_r(X)\mid \bigcup_i D_i\in\UD_r(X)\}.
     \end{align*}
     Elements of $\UD^I_r(X)$ are called \emph{$r$-uniformly discrete multi sets}.
      Elements of
       \begin{align*}
	     \UD^I(X)=\bigcup_{r>0}\UD^I_r(X)
       \end{align*}
       are called \emph{uniformly discrete multi sets}.
      A uniformly discrete multi set $(D_i)_i\in\UD^I(X)$ is called a
      \emph{Delone multi set}
     if each $D_i$ is a Delone set and the union $\bigcup_iD_i$ is a Delone set.

            A cutting-off operation on $\UD^I(X)$ is defined by regarding it as a subspace
     of the product space $\prod_{i\in I}\UD(X)$, that is, via a equation
      \begin{align*}
             (D_i)_{i\in I}\sci C=(D_i\cap C)_{i\in I}.
      \end{align*}
       By this operation $\UD^I(X)$ and $\UD_r^I(X)$ are abstract pattern spaces over $X$.
\end{ex}





\subsection{An order on abstract pattern spaces}
\label{subsection_order_pat_sp}
Here we discuss an order on abstract pattern spaces. All the proofs are found in
\cite{Nagai3rd}.

\begin{defi}\label{def_order_abstract pattern_space}
     Let $\Pi$ be an abstract pattern space over $X$. We define a relation
      $\geqq$ on $\Pi$ as follows:
     for each $\calP,\calQ\in\Pi$, we set $\calP\geqq\calQ$ if 
     \begin{align*}
          \calP\sci\supp\calQ=\calQ.
     \end{align*}
\end{defi}

\begin{lem}
     \begin{enumerate}
      \item If $\calP\geqq\calQ$, then $\supp\calP\supset\supp\calQ$.
      \item The relation $\geqq$
             is an order on $\Pi$.
     \end{enumerate}
\end{lem}

\begin{lem}\label{lem_order_abstract pattern_space}
     \begin{enumerate}
      \item If $\calP\in\Pi$ and $C\in\calC(X)$, then $\calP\geqq\calP\sci C$.
      \item If $\calP,\calQ\in\Pi$, $C\in\calC(X)$ and $\calP\geqq\calQ$, then
            $\calP\sci C\geqq\calQ\sci C$.
     \end{enumerate}      
\end{lem}



The supremum of $\Xi\subset\Pi$ with respect to this order $\geqq$ describes
``the union'' of $\Xi$ that is obtained by ``gluing'' elements of $\Xi$. We will discuss
this gluing operation in the next subsection.

\begin{defi}
       Let $\Xi$ be a subset of an abstract pattern space $\Pi$.
       If the supremum of $\Xi$ with respect to the order $\geqq$ defined in 
       Definition \ref{def_order_abstract pattern_space}
exists in $\Pi$,
       it is denoted by $\bigvee\Xi$.
\end{defi}

The following lemma describes a relation between $\bigvee$ and supports.
\begin{lem}\label{lem_support_supremum}
       If a subset $\Xi\subset\Pi$ admits the supremum $\bigvee\Xi$, then
       $\supp\bigvee\Xi=\overline{\bigcup_{\calP\in\Xi}\supp\calP}$.
\end{lem}
\begin{rem}
        It is not necessarily true that any element $\calP_{0}$ in $\Pi$
        that majorizes $\Xi$ and $\supp\calP_{0}=\overline{\bigcup_{\calP\in\Xi}\supp\calP}$
        is the supremum of $\Xi$ (\cite{Nagai3rd}).
\end{rem}

\subsection{Glueable abstract pattern spaces}
\label{subsection_glueable_pat_sp}

        In this subsection $\Pi$ is an abstract pattern space over $X$.

        Often we want to ``glue'' abstract patterns to obtain a larger abstract pattern.
        For example, suppose
         $\Xi$ is a collection of patches such that if $\calP,\calQ\in\Xi$,
         $S\in\calP$ and $T\in\calQ$, then we have either $S=T$ or $S\cap T=\emptyset$.
         Then we can ``glue'' patches in $\Xi$, that is, we can take 
         the union $\bigcup_{\calP\in\Xi}\calP$, which is also a patch.
         Abstract pattern spaces in which we can often
	 ``glue'' abstract patterns are called
         glueable abstract pattern spaces (Definition \ref{def_glueable_abstract pattern_space}).  This gluing operation is essential when we construct the limit of a
	 Cauchy sequence
	 (Theorem \ref{thm_sigma_complete_general_ver}).

	 We first introduce notions that are used to define ``glueable'' abstract
	 pattern spaces, where we can often ``glue'' abstract patterns.
\begin{defi}\label{def_local_finite_compatible}
      \begin{enumerate}
       \item Two abstract patterns $\calP,\calQ\in\Pi$ are said to be \emph{compatible}
              if there is $\calR\in\Pi$ such that $\calR\geqq\calP$ and
              $\calR\geqq\calQ$.
       \item A subset $\Xi\subset\Pi$ is said to be \emph{pairwise compatible} if
             any two elements $\calP,\calQ\in\Pi$ are compatible.
        \item A subset $\Xi\subset\Pi$ is said to be \emph{locally finite}
               if for any $x\in X$ and $r>0$, the set $\Xi\sci B(x,r)$, which is defined
	      via
	      \begin{align*}
	       \Xi\sci B(x,r)=\{\calP\sci B(x,r)\mid\calP\in\Xi\},
	      \end{align*}
               is finite.
      \end{enumerate}
\end{defi}

\begin{rem}
      We will see for many examples of abstract pattern spaces $\Pi$, a locally finite
 and pairwise compatible $\Xi\subset\Pi$ admits the supremum. We have to assume being
 pairwise compatible because if $\Xi\subset\Pi$ admits the supremum, any
 $\calP,\calQ\in\Xi$ are compatible (we can use the supremum for the role of $\calR$
 above). We have to assume local finiteness because without this there is a
 counterexample that do not admits the supremum (see \cite{Nagai3rd}).
\end{rem}

We use the following lemma to define glueable abstract pattern spaces.
\begin{lem}\label{lem_locally_finite_compatible}
        Let $\Xi$ be a subset of $\Pi$ and take $C\in\calC(X)$.
        Then the following hold.
        \begin{enumerate}
	 \item If $\Xi$ is locally finite, then so is $\Xi\sci C$, which is defined via
	       \begin{align*}
		    \Xi\sci C =\{\calP\sci C\mid\calP\in\Xi\}.
	       \end{align*}
         \item If $\Xi$ is pairwise compatible, then so is $\Xi\sci C$.
	\end{enumerate}            
\end{lem}


\begin{defi}\label{def_glueable_abstract pattern_space}
          $\Pi$ is said to be glueable
          if the following two conditions hold:
          \begin{enumerate}
	   \item If $\Xi\subset\Pi$ is both locally finite and pairwise compatible,
                 then there is the supremum $\bigvee{\Xi}$ for $\Xi$.
           \item If $\Xi\subset\Pi$ is both locally finite and pairwise compatible,
                 then for any $C\in\calC(X)$,
                 \begin{align}
		     \bigvee(\Xi\sci C)=(\bigvee\Xi)\sci C.\label{eq_def_glueable}
		 \end{align}
	  \end{enumerate}
\end{defi}

\begin{rem}
       By Lemma \ref{lem_locally_finite_compatible}, for $\Xi\subset\Pi$
      which is locally finite and pairwise compatible 
       and $C\in\calC(X)$
       the left-hand side of the equation (\ref{eq_def_glueable}) makes sense.
\end{rem}

We finish this subsection with examples. The details are found in \cite{Nagai3rd}.
\begin{ex}
      Consider $\Pi=\Patch(X)$ (Example \ref{example_patch}).
      In this abstract pattern space, for two elements $\calP,\calQ\in\Patch(X)$,
       the following statements hold:
      \begin{enumerate}
       \item  $\calP\geqq\calQ\iff\calP\supset\calQ$.
       \item  $\calP$ and $\calQ$ are compatible if and only if for any $T\in\calP$ and
              $S\in\calQ$, either $S=T$ or $S\sci T=\emptyset$ holds.
      \end{enumerate}
       If $\Xi\subset\Patch(X)$ is pairwise compatible, then
       $\calP_{\Xi}=\bigcup_{\calP\in\Xi}\calP$ is a patch,
        which is the supremum of $\Xi$. If $C\in\calC(X)$, then
        \begin{align*}
	      (\bigvee\Xi)\sci C=(\bigcup_{\calP\in\Xi}\calP)\sci C
             =\bigcup(\calP\sci C)=\bigvee(\Xi\sci C).
	\end{align*}    
         $\Patch(X)$ is glueable.
\end{ex}

\begin{ex}
       For the abstract pattern space $2^X$ in Example \ref{example_2^X_calC(X)},
        two elements $A,B\in 2^X$
       are compatible if and only if 
       \begin{align*}
        \text{ $\overline{A}\cap B\subset A$ and
       $A\cap\overline{B}\subset B$.}
       \end{align*}

       Suppose $\Xi\subset 2^X$ is locally finite and pairwise compatible.
       Note that $\bigcup_{A\in\Xi}\overline{A}=\overline{\bigcup_{A\in\Xi}A}$.
       Set $A_{\Xi}=\bigcup_{A\in\Xi}A$. For each $A\in\Xi$,
       $A_{\Xi}\cap \overline{A}=\bigcup_{B\in\Xi}(B\cap\overline{A})=A$;
       $A_{\Xi}$ is a majorant of $\Xi$. If $B$ is also a majorant for $\Xi$,
      then
       \begin{align*}
	B\cap\overline{A_{\Xi}}=B\cap(\bigcup_{A\in\Xi}\overline{A})
                              =\bigcup_{A\in\Xi}(B\cap\overline{A})
                             =\bigcup_{A\in\Xi}A=A_{\Xi},
       \end{align*}
        and so $B\geqq A_{\Xi}$.
       It turns out that $A_{\Xi}$ is the supremum for $\Xi$.
       Moreover, if $C\in\calC(X)$, then
      $A_{\Xi}\sci C=\bigcup_{A\in\Xi}(A\cap C)=\bigvee(\Xi\cap C)$.
      Thus  $2^X$ is a glueable space.  
\end{ex}



 \begin{rem}
       Let $\Pi_0$ be a glueable abstract pattern space and $\Pi_1\subset\Pi_0$ an
       abstract pattern subspace. For any subset $\Xi\subset\Pi_1$,
       if  it is pairwise compatible
      in $\Pi_1$, then it is pairwise compatible in $\Pi_0$.  
      Moreover, whether a set is locally finite or not is independent of the
      ambient abstract pattern space in which the set is included.
      For a subset $\Xi\subset\Pi_1$ which is locally finite and
     pairwise compatible in $\Pi_1$, since $\Pi_0$ is glueable,
     there is the supremum $\bigvee\Xi$ in $\Pi_0$.
     If this supremum in $\Pi_0$ is always included in 
     $\Pi_1$, then $\Pi_1$ is glueable.

     By this remark it is easy to see the abstract pattern spaces $\calC(X)$ 
     (Example \ref{example_2^X_calC(X)}), $\LF(X)$
     (Example \ref{example_LF(X)}), and $\UD_r(X)$ (Example
      \ref{example_UD},
      $r$ is an arbitrary positive number) are
      glueable.

      However, $\UD(X)$ (Example \ref{example_UD}) is not necessarily glueable.
      For example, set $X=\mathbb{R}$.
      Set $\calP_n=\{n,n+\frac{1}{n}\}$ for each integer $n\neq 0$.
      Each $\calP_n$ is in $\UD(\mathbb{R})$, $\Xi=\{\calP_n\mid n\neq 0\}$
      is locally finite and pairwise compatible, but it does not admit the supremum.
 \end{rem}

\begin{lem}\label{lem_UDIr_glueable}
      Let $I$ be a non-empty set and $r$ be a positive real number. The abstract
      pattern space $\UD_r^I(X)$ (Example \ref{ex_multi_set}) is glueable.
\end{lem}
\begin{proof}
      Let $p_i\colon\UD_r^I(X)\rightarrow \UD_r(X)$ be the projection to $i$-th  element.
      For two $D,E\in\UD_r^I(X)$, if they are compatible, there is $F\in\UD_r^I(X)$
      such that $F\geqq D$ and $F\geqq E$. By the definition of cutting-off operation,
     we have
 $p_i(F)\cap \supp D=p_i(D)$ and $p_j(F)\cap \supp E=p_j(E)$
      for each $i,j\in I$, which means if $x\in p_i(D)$ and $y\in p_j(E)$, then
 we have either $x=y$ or $\rho_X(x,y)\geqq r$.
     We also have $p_i(D)\cap\supp E\subset p_i(F)\cap \supp E=p_i(E)$ for each $i$.

      Let $\Xi$ be a pairwise compatible subset of $\UD_r^I(X)$. For each $i$ set
      $D_i=\bigcup_{E\in\Xi}p_i(E)$. By the above observation, the tuple
      $D=(D_i)_{i\in I}$ is an element of $\UD_r^I(X)$. The above observation also implies
      that for each $E\in\Xi$, we have $D_i\cap\supp E=p_i(E)$, which means $D\geqq E$.
       We see $D$ is a majorant of $\Xi$, and
     since $\supp D=\bigcup_{E\in\Xi}\supp E$, we  see $D$ is the supremum.
\end{proof}
We mention the following proposition without proving it.
\begin{prop}
      If $Y$ is a non-empty set and $y_0\in Y$, the abstract pattern space
      $\Map(X,Y,y_0)$ (Example \ref{example_map})is glueable.
\end{prop}

\subsection{Zero Element and Its Uniqueness}
\label{subsection_zero_elements}
In this subsection we discuss zero elements, which is on the list of structures
at the beginning of this section.
\begin{defi}
      Let $\Pi$ be an abstract pattern space over $X$.
      An element $\calP\in\Pi$ such that $\supp\calP=\emptyset$ is called
      a zero element of $\Pi$. If there is only one zero element in $\Pi$, 
      it is denoted by $0$.
\end{defi}

\begin{rem}
      Zero elements always exist. In fact, take an arbitrary element $\calP\in\Pi$.
      Then by Lemma \ref{lemma_support_calP_sci_C}, $\supp(\calP\sci\emptyset)=\emptyset$
       and so $\calP\sci\emptyset$ is a zero element.
\end{rem}

\begin{lem}\label{uniqueness_zero_element}
        If $\Pi$ is a glueable abstract pattern space over $X$,
 there is only one zero element in $\Pi$.
\end{lem}
\begin{proof}
        The subset  $\emptyset$ of $\Pi$ is locally finite and pairwise compatible.
        Set $\calP=\bigvee\emptyset$.
       By Lemma \ref{lem_support_supremum}, $\calP$ is a zero element.
       If $\calQ$ is a zero element, then since
        $\calQ$ is a majorant for $\emptyset$, we see $\calQ\geqq\calP$.
        We have $\calQ=\calQ\sci\emptyset=\calP$.
\end{proof}

\begin{lem}\label{lem_supremum_Xi_and_Xi_zero}
        Let $\Pi$ be a glueable abstract pattern space over $X$.
        Take a locally finite and pairwise compatible subset $\Xi$ of $\Pi$.
        Then $\bigvee\Xi\cup \{0\}$ exists and
        $\bigvee\Xi\cup\{0\}=\bigvee\Xi$.
\end{lem}

\subsection{$\Gamma$-abstract pattern spaces over $X$, or abstract pattern spaces over $(X,\Gamma)$}
\label{subsection_gammma-abstract pattern_space}
Here we incorporate group actions to the theory of abstract pattern spaces.
First we define abstract pattern spaces over $(X,\Gamma)$, or $\Gamma$-abstract pattern spaces over $X$.
We require there is an action of the group $\Gamma$ on such an abstract pattern space and
the cutting-off operation is
equivariant.

       In this subsection,  $\Pi$ is an
       abstract pattern space over $X$.

\begin{defi}\label{def_gamma-abstract pattern_space}
       Suppose there is a group action $\Gamma\curvearrowright\Pi$ such that 
       for each $\calP\in\Pi, C\in\calC(X)$ and $\gamma\in\Gamma$,
       we have $(\gamma\calP)\sci (\gamma C)=\gamma(\calP\sci C)$, that is, the
 cutting-off operation
       is equivariant.
       Then we say $\Pi$ is a \emph{$\Gamma$-abstract pattern space over $X$}
       or a \emph{abstract pattern space over $(X,\Gamma)$}.
        For an
        abstract pattern space $\Pi$ over $(X,\Gamma)$, its nonempty subset $\Sigma$
       such that $\calP\in\Sigma$ and $\gamma\in\Gamma$ imply $\gamma\calP\in\Sigma$
      is called a \emph{subshift} of $\Pi$.
\end{defi}

We first investigate the relation between the group action and the two construction of abstract pattern spaces,
taking subspace and taking product.
\begin{lem}\label{lem_abstract pattern_subspace_over_X_Gamma}
     Let $\Pi$ be an abstract pattern space over $(X,\Gamma)$.
       Suppose $\Pi'$ is an abstract pattern subspace of
      $\Pi$. If $\Pi'$ is closed under the $\Gamma$-action, then $\Pi'$ is an abstract pattern space
       over $(X,\Gamma)$.
\end{lem}

 \begin{lem}\label{lem_product_Gamma-abstract pattern_space}
       Let $\Lambda$ be a set and $(\Pi_{\lambda})_{\lambda\in\Lambda}$ be a family of
       abstract pattern spaces over $(X,\Gamma)$.
       Then $\Gamma$ acts on the product space $\prod_{\lambda}\Pi_{\lambda}$ by 
       $\gamma(\calP_{\lambda})_{\lambda}=(\gamma\calP_{\lambda})_{\lambda}$ and
      by this action $\prod_{\lambda}\Pi_{\lambda}$ is an abstract pattern space over $(X,\Gamma)$.
 \end{lem}

\begin{defi}\label{def_product_Gamma_abstract pattern_spd}
       The abstract pattern space $\prod\Pi_{\lambda}$ is called the product $\Gamma$-abstract pattern space.
\end{defi}


We now list several examples of $\Gamma$-abstract pattern spaces.

\begin{ex}\label{ex_patch_Gamma_abstract pattern_sp}
     For $\calP\in\Patch(X)$ and $\gamma\in\Gamma$,
      set $\gamma\calP=\{\gamma T\mid T\in\calP\}$.
     This defines an action of $\Gamma$ on $\Patch(X)$ and makes $\Patch(X)$ an
     abstract pattern space over $(X,\Gamma)$.

     For an $L$-labeled tile $(T,l)$ and $\gamma\in\Gamma$, set
      $\gamma(T,l)=(\gamma T,l)$. This defines an action of $\Gamma$ on $\Patch_L(X)$
       (Example \ref{ex_L-labeled_tiling}), which makes it a $\Gamma$-abstract pattern
     space.
\end{ex}

\begin{ex}\label{2X_as_Gamma_abstract pattern_sp}
     $2^X$ (Example \ref{example_2^X_calC(X)}) is an abstract pattern space over $(X,\Gamma)$,
    by the action $\Gamma\curvearrowright 2^X$ inherited from the action
     $\Gamma\curvearrowright X$.
    By Lemma \ref{lem_abstract pattern_subspace_over_X_Gamma},  the spaces $\LF(X)$(Example \ref{example_LF(X)}),  $\calC(X)$ (Example \ref{example_2^X_calC(X)}),
     $\UD(X)$ 
 and $\UD_r(X)$ (Example \ref{example_UD}, $r>0$) are all
     abstract pattern spaces over $(X,\Gamma)$.

      The abstract pattern space $\UD^I(X)$ (Example \ref{ex_multi_set}) is also
      an $\Gamma$-abstract pattern space.
\end{ex}

\begin{ex}\label{ex_map_rho}
       Take a non-empty set $Y$, an element $y_0\in Y$ and
      and an action $\phi\colon\Gamma\curvearrowright Y$ that fixes $y_0$.
      As was mentioned before (Example \ref{example_map}), 
       $\Map (X,Y,y_0)$ is an abstract pattern space over $X$.       
       Define an action of $\Gamma$ on $\Map(X,Y,y_0)$ by
        \begin{align*}
	       (\gamma f)(x)=\phi(\gamma)(f(\gamma^{-1}x)).
	\end{align*}
        By this group action $\Map(X,Y,y_0)$ is $\Gamma$-abstract pattern space.
         This $\Gamma$-abstract pattern space is denoted by $\Map_{\phi}(X,Y,y_0)$.
         If $\phi$ sends every group element to the identity, we denote the corresponding
         space by $\Map(X,Y,y_0)$.
          This group action is essential when we study pattern-equivariant functions,
         since if $\calP$ is an abstract pattern,
  a function $f\in\Map_{\phi}(X,Y,y_0)$ is $\calP$-equivariant if and only
 if $f$ is locally derivable from $\calP$ (\cite{Nagai3rd}).
\end{ex}

\begin{ex}\label{ex_Gamma-abstract pattern_space_measures}
       The dual space $C_c(X)'$ is an abstract pattern space over $X$
       (Example \ref{ex_space_of_measures}).
       For $\varphi\in C_c(X)$ and $\gamma\in\Gamma$, set
      $(\gamma\varphi)(x)=\varphi(\gamma^{-1}x)$.
       For $\Phi\in C_c(X)'$ and $\gamma\in\Gamma$, set
       $\gamma\Phi(\varphi)=\Phi(\gamma^{-1}\varphi)$.
       Then $C_c(X)'$ is an abstract pattern space over $(X,\Gamma)$.

       Let $r$ be a positive real number.
       The space $\WDC_r(X)$ of all weighted Dirac combs $\sum_{x\in D}w(x)\delta_x$,
       where $D$ is an $r$-uniformly discrete subset of $X$,
       $w\colon D\rightarrow\mathbb{C}\setminus\{0\}$ and $\delta_x$ is the Dirac measure
       on $x$, is a abstract pattern space over $(X,\Gamma)$, which is a subshift of
       $C_c(X)'$.
\end{ex}

We mention two examples of subshifts.
\begin{ex}\label{ex_subshift_of_Delone_sets}
        The set $\Del(X)$ of all Delone sets in $X$ is a subshift of $\UD(X)$
        (Example \ref{example_UD}).
\end{ex}

\begin{ex}
 The space of all tilings is a subshift of $\Patch(X)$ (Example \ref{example_patch}).
\end{ex}

In the previous subsection, we defined glueable abstract pattern spaces.
Here, we define corresponding notions for $\Gamma$-abstract pattern spaces and
subshifts.

\begin{defi}
        Assume $\Pi$ is an abstract pattern space over $(X,\Gamma)$.
        We say $\Pi$ is a \emph{glueable abstract pattern space over $(X,\Gamma)$}
         if it is a glueable
        abstract pattern space over $X$.
        For a glueable abstract pattern space $\Pi$, its subshift $\Sigma$ is said to be 
        \emph{supremum-closed}
         if for any pairwise compatible and locally finite $\Xi\subset\Sigma$,
        we have $\bigvee\Xi\in\Sigma$.
\end{defi}

\begin{rem}
      In the definition of supremum-closed subshifts, the supremum $\bigvee\Xi$ exists
      by the definition of glueable abstract pattern space.
\end{rem}

\section{The definition and properties of local matching topology}
\label{subsection_def_local_matching_top}

\begin{nota}
       Let $\Cpt(X)$ be the set of all compact subsets of $X$ and
       $\mathscr{V}$ the set of all compact neighborhoods of $e\in\Gamma$
       (the identity element of $\Gamma$).
\end{nota}

In this section $\Pi$ is an abstract pattern space over $(X,\Gamma)$.

In this section we define and investigate the local matching topology on $\Pi$.
We use the theory of uniform structure to define them.
(For the theory of uniform structures, see \cite{MR1726779}.)
The uniform structure will be metrizable (Corollary \ref{102557_9Aug18}),
but the description of a metric is
not simple when $\Gamma$ is non-commutative, and this is why we prefer uniform structures.
With respect to this uniform structure, two abstract patterns $\calP$ and $\calQ$ in $\Pi$
are ``close'' when they match in a ``large region'' after sliding $\calQ$ by 
a ``small'' $\gamma\in\Gamma$.
This is analogous to the product topology of the space $\calA^{\mathbb{Z}}$, where
$\calA$ is a finite set; in fact we can show on this space the relative topology of
the local matching topology on a space of maps $\Map(\mathbb{Z},\calA\cup\{*\})$
($*$ is a point outside $\calA$)
 coincides with the product topology.

Here is the plan of this section. In Subsection \ref{subsection_def_LMT} we define
the local matching uniform structure and topology.
In Subsection \ref{subsection_hausdorff_property} we give a sufficient condition
for the local matching topology on a subshift to be Hausdorff, and prove many examples
of subshifts satisfies this condition.
In Subsection \ref{Subsection_complete} we give a mild condition that assures that
the local matching uniform structure is complete, and show if the action
$\Gamma\curvearrowright X$ is proper, the usual definition of FLC implies the compactness
of the continuous hull (the closure of the orbit).
 
\subsection{The definition of local matching topology}
\label{subsection_def_LMT}
We first define the following notation, which will be used to define a uniform structure.
\begin{defi}
        For $K\in\Cpt(X)$ and $V\in\mathscr{V}$, set
        \begin{align*}
	      \calU_{K,V}=\{(\calP,\calQ)\in\Pi\times\Pi\mid 
                 \text{there is $\gamma\in V$ such that 
                $\calP\sci K=(\gamma \calQ)\sci K$}\}.
	\end{align*}
\end{defi}

We first remark the following lemma, the proofs of which are easy.
\begin{lem}\label{lem_for_UKV_fund_sys_entourage}
       If $K_1\subset K_2$ and $V_2\subset V_1$, then
       $\calU_{K_2,V_2}\subset\calU_{K_1,V_2}$.
\end{lem}

We define a uniform structure by constructing a filter basis on $\Pi\times\Pi$ that
satisfies the axiom of fundamental system of entourages (\cite[II.2,\S 1,1.]{MR1726779}).
\begin{lem}\label{lem_UKV_fundamental_system_entourage}
       The set 
       \begin{align}
	     \{\calU_{K,V}\mid K\in\Cpt(X), V\in\mathscr{V}\}
             \label{fundamental_system_entourage}
       \end{align}
        satisfies the axiom of fundamental system of entourages.
\end{lem}
\begin{proof}
       (1) For any $K\in\Cpt(X)$, $V\in\mathscr{V}$, we show
         \begin{align*}
	      \{(\calP,\calP)\mid\calP\in\Pi\}\subset\calU_{K,V}.
	 \end{align*}
         This is clear since if $\calP\in\Pi$, we have
          $\calP\sci K=\calP\sci K$.

       (2) For any $K$ and $V$, we show there are $K'$ and $V'$ such that
        \begin{align*}
	    \calU_{K',V'}\subset\calU_{K,V}^{-1}=\{(\calQ,\calP)\mid (\calP,\calQ)\in
	 \calU_{K,V}\}.
	\end{align*}
         Take $(\calP,\calQ)\in\calU_{V^{-1}K,V^{-1}}$.
           There is $\gamma\in V$ such that
          \begin{align*}
	       \calP\sci V^{-1}K=(\gamma^{-1}\calQ)\sci V^{-1}K.
	  \end{align*}
           Multiplying by $\gamma$ both sides we have
           \begin{align*}
	        (\gamma\calP)\sci\gamma V^{-1}K=\calQ\sci\gamma V^{-1}K, 
	   \end{align*}
           and so 
          \begin{align*}
	        (\gamma\calP)\sci K&=(\gamma\calP)\sci\gamma V^{-1}K\sci K\\
                                   &=\calQ\sci \gamma V^{-1}K\sci K\\
                                  &=\calQ\sci K.
	  \end{align*}
          We have $(\calQ,\calP)\in\calU_{K,V}$ and so 
          $\calU_{V^{-1}K,V^{-1}}^{-1}\subset\calU_{K,V}$.

          (3) For each $K_1,K_2\in\Cpt(X)$ and $V_1,V_2\in\mathscr{V}$, we show
        there are $K_3\in\Cpt(X)$ and $V_3\in\mathscr{V}$ such that
        \begin{align*}
	       \calU_{K_3,V_3}\subset\calU_{K_1,V_1}\cap\calU_{K_2,V_2}.
	\end{align*} 
 By Lemma \ref{lem_for_UKV_fund_sys_entourage}, for $K_1,K_2\in\Cpt(X)$
          and $V_1,V_2\in\mathscr{V}$, we have
          \begin{align*}
            \calU_{K_1\cup K_2,V_1\cap V_2}\subset\calU_{K_1,V_1}\cap \calU_{K_2,V_2}.
	  \end{align*}

          (4)Take $K\in\Cpt(X)$ and $V\in\mathscr{V}$ arbitrarily.
         We show there are $K'$ and $V'$ such that
         \begin{align*}
	        \calU_{K',V'}^2=\{(\calP,\calR)\mid\text{there is $\calQ$ with
	        $(\calP,\calQ),(\calQ,\calR)\in\calU_{K',V'}$}\}\subset\calU_{K,V}.
	 \end{align*}
          Set $K_1=(V^{-1}K)\cup K$ and take $V_1\in\mathscr{V}$ such that
           $V_1V_1\subset V$. Note that $V_1\subset V$.
          If $(\calP_1,\calP_2),(\calP_2,\calP_3)\in\calU_{K_1,V_1}$, then there are
          $\gamma_1$ and $\gamma_2$ in $V_1$ such that
          $\calP_1\sci K_1=(\gamma_1\calP_2)\sci K_1$ and
          $\calP_2\sci K_1=(\gamma_2\calP_3)\sci K_1$. We have
          \begin{align*}
	        (\gamma_1\gamma_2\calP_3)\sci K&=
                      \gamma_1((\gamma_2\calP_3)\sci K_1)\sci K\\
                       &=\gamma_1(\calP_2\sci K_1)\sci K\\
                       &=(\gamma_1\calP_2)\sci K \\
                      &=((\gamma_1\calP_2)\sci K_1)\sci K\\
                      &=(\calP_1\sci K_1)\sci K\\
                     &=\calP_1\sci K.
	  \end{align*}
          Thus $(\calP_1,\calP_3)\in\calU_{K,V}$.
         We have proved $\calU_{K_1,V_1}^2\subset\calU_{K,V}$.
\end{proof}

\begin{defi}
      Let $\mathfrak{U}$ be the set of all entourages generated by
      (\ref{fundamental_system_entourage}). That is, $\mathfrak{U}$ is the set of all
 subsets $\calU$ of $\Pi\times\Pi$ such that there are $K\in\Cpt(X)$ and $V\in\mathscr{V}$
     with $\calU\supset\calU_{K,V}$.
      The uniform structure defined by $\mathfrak{U}$ is called the local matching
      uniform structure. The topology defined by this uniform structure, that is, the
       topology in which the set
        \begin{align*}
	      \{\calU_{K,V}(\calP)\mid K\in\Cpt(X), V\in\mathscr{V}\}
	\end{align*}
         is a fundamental system of neighborhoods for $\calP\in\Pi$,
         is called the local matching topology. Here,
        \begin{align*}
	     \calU_{K,V}(\calP)=\{\calQ\in\Pi\mid(\calP,\calQ)\in\calU_{K,V}\}.
	\end{align*}
\end{defi}

\subsection{Hausdorff property of local matching topology}
\label{subsection_hausdorff_property}
Next we give a sufficient condition for the local matching topology to be Hausdorff in
Proposition \ref{prop_sufficint_condition_hausdorff}.

\begin{defi}
      Suppose $\Pi$ admits a unique zero element $0$.
      An abstract pattern $\calP\in\Pi$ is called an atom if $\supp\calP$ is
     compact and
     \begin{align*}
        \text{$ \calQ\in\Pi$ and $\calQ\leqq\calP\Rightarrow\calP=\calQ$ or $\calQ=0$.}
     \end{align*}
      For $\calP\in\Pi$ set
      \begin{align*}
           A(\calP)=\{\text{$\calQ\colon$ atom}\mid \calQ\leqq\calP\}.
      \end{align*}
      A subset $\Sigma\subset\Pi$ is said to be atomistic if
      for any $\calP\in\Pi$ we have $\calP=\bigvee A(\calP)$.

      A subset $\Sigma\subset\Pi$ is said to have limit inclusion property
      if the following condition is satisfied:
      \begin{align*}
           \text{for any $\calP\in\Sigma$ and an atom $\calQ\in\Pi$, if for any
           $V\in\mathscr{V}$}\\
          \text{there is $\gamma_V\in V$
            such that $\gamma_V\calQ\leqq\calP$, we have $\calQ\leqq\calP$.}
      \end{align*}
\end{defi}

\begin{prop}\label{prop_sufficint_condition_hausdorff}
      Suppose $\Pi$ admits a unique zero element $0$.
      Let $\Sigma$ be a nonempty subset of $\Pi$ which is atomistic and
       has limit inclusion property.
       Then the local matching topology on $\Sigma$ is Hausdorff.
\end{prop}
\begin{proof}
      Take $\calP,\calQ\in\Sigma$ and suppose $(\calP,\calQ)\in\mathcal{U}_{K,V}$
      for any $K\in\Cpt(X)$ and $V\in\mathscr{V}$.
      We show $\calP=\calQ$.
      Take $\calR\in A(\calP)$. Set $K=\supp\calR$. For any $V\in\mathscr{V}$
      there is $\gamma_V\in V$ such that
      \begin{align*}
            \calP\sci K=(\gamma_V^{-1}\calQ)\sci K.
      \end{align*}
      This implies that
       \begin{align*}
	     \gamma_V\calR\leqq\calQ,
       \end{align*}
      and so by limit inclusion property, we have
       \begin{align*}
	     \calR\leqq\calQ.
       \end{align*}
       Since $\Sigma$ is atomistic, we have $\calP\leqq\calQ$.
       The converse is proved in the same way and we have
      $\calQ\leqq\calP$, and so $\calP=\calQ$.
\end{proof}

\begin{cor}\label{102557_9Aug18}
       Under the same assumption on $\Pi$ and $\Sigma$, the local matching topology
      on $\Sigma$ is metrizable.
\end{cor}
\begin{proof}
       This follows from  Proposition \ref{prop_sufficint_condition_hausdorff},
 \cite[IX, \S 2]{MR1726872} and
 the fact that $\{\calU_{K_n,V_n}\mid n=1,2,\ldots\}$, where
        $K_n=B(x_0,R_n), V_n=B(e,r_n), R_n\nearrow\infty$ and $r_n\searrow 0$, forms a
       fundamental system of entourages.
\end{proof}

We then give examples of sets of abstract patterns on which the local matching topology
is Hausdorff, by checking the conditions in Proposition \ref{prop_sufficint_condition_hausdorff}.
\begin{lem}\label{lem_Cb_atomistic_LIP}
      Let $Y$ be a non-empty topological space and $y_0$ be an element of $Y$.
      Take a group homomorphism $\phi\colon\Gamma\rightarrow\Homeo(Y)$ which is
       continuous with respect to the compact-open topology and such that
       $\phi(\gamma)y_0=y_0$ for each $\gamma\in\Gamma$.
      Then $C_b(X,Y,y_0)=\{f\in\Map_{\phi}(X,Y,y_0)\mid \text{continuous and bounded}\}$
     is atomistic and has limit inclusion property
      as a subset of the abstract pattern space $\Map_{\phi}(X,Y,y_0)$
     (Definition \ref{ex_map_rho}).
\end{lem}
\begin{proof}
      For each $x\in X$ and $y\in Y\setminus\{y_0\}$, the function defined by
      \begin{align*}
           \varphi_x^y(x')=
            \begin{cases}
	         y &\text{if $x=x'$}\\
	         y_0 &\text{if $x\neq x'$}
	    \end{cases}
      \end{align*}
      is an atom of $\Map_{\phi}(X,Y,y_0)$.
      Any atom of $\Map_{\phi}(X,Y,y_0)$ is of this form.
      For $f\in C_b(X,Y,y_0)$, we have
       \begin{align*}
	   A(f)=\{\varphi_x^{f(x)}\mid\text{$x\in X$ and $f(x)\neq y_0$}\}.
       \end{align*}
       We see $f=\bigvee A(f)$. We have proved that $C_b(X,Y,y_0)$ is
      atomistic.

      Next we show $C_b(X,Y,y_0)$ has limit inclusion property.
      Take any $x\in X$ and $y\in Y\setminus\{y_0\}$, and assume that for any
      $V\in\mathscr{V}$ there is $\gamma_V\in V$ such that
      $\gamma_V\varphi_x^{y}\leqq f$.
       Since $\supp\gamma_V\varphi_x^y=\{\gamma_Vx\}$, we have
       \begin{align*}
	    f(\gamma_Vx)=(\gamma_V\varphi_x^y)(\gamma_Vx)=\phi(\gamma_V)(\varphi_x^y(x))
	    =\phi(\gamma_V)(y).
       \end{align*}
       Since $f$ is continuous and the action $\Gamma\curvearrowright X$ is continuous,
      \begin{align*}
            f(x)=\lim_Vf(\gamma_V x)=\lim_V\phi(\gamma_V)(y)=y,
      \end{align*}
      and so $f\geqq\varphi_x^y$.
      We have shown $C_b(X,Y,y_0)$ has limit inclusion property.
\end{proof}

\begin{cor}
     The relative topology of the local matching topology on $C_b(X,Y,y_0)$
     is Hausdorff.
\end{cor}
\begin{proof}
    Clear by Proposition \ref{prop_sufficint_condition_hausdorff}
    and Lemma \ref{lem_Cb_atomistic_LIP}.
\end{proof}

As the following lemma shows, the local matching topology on $\Map_{\phi}(X,Y,y_0)$
is not necessarily Hausdorff:
\begin{lem}
    On  $\Map (\mathbb{R},\mathbb{C},0)$, the local matching topology is not
    Hausdorff.
\end{lem}
\begin{proof}
     Take characteristic functions
 $f=1_{\mathbb{Q}}$ and $g=1_{\mathbb{Q}+a}$, where $a$ is any irrational
    number.
    Then $(f,g)$ belongs to any entourage.
\end{proof}

We then prove on $\Patch(X)$ (Example \ref{example_patch}, Example \ref{ex_patch_Gamma_abstract pattern_sp}),
the local matching topology is Hausdorff.

\begin{lem}\label{lem_patch_atomistic_LIP}
      The abstract pattern space $\Patch(X)$
      over $(X,\Gamma)$ is atomistic and
      has limit inclusion property.
\end{lem}
\begin{proof}
     Let $T$ be a tile. Then $\{T\}$ is an atom.
     Any atom in $\Patch(X)$ is of this form.
    For any patch $\calP\in\Patch(X)$, we have
    \begin{align*}
         A(\calP)=\{\{T\}\mid T\in\calP\},
    \end{align*}
    and so $\calP=\bigcup A(\calP)=\bigvee A(\calP)$.
     We have shown that $\Patch(X)$ is atomistic.

     To prove $\Patch(X)$ satisfies limit inclusion property,
     take $\calP\in\Patch(X)$ and a tile $T$, and assume for any $V\in\mathscr{V}$
     there is $\gamma_V\in V$ such that $\gamma_V\{T\}\leqq\calP$, that is,
     $\gamma_V T\in\calP$.
     We show $T\in\calP$.
     There is $V_0\in\mathscr{V}$ such that if $V_1,V_2\in\mathscr{V}$ and
      $V_j\subset V_0$ for each $j$, then $\gamma_{V_1}T\cap\gamma_{V_2}T\neq
      \emptyset$. Since $\gamma_{V_j}T$ is in a patch $\calP$ for each $j$,
      we see $\gamma_{V_1}T=\gamma_{V_2}T$.
      Now it suffices to show that $T=\gamma_{V_0}T$ since $\gamma_{V_0}T\in\calP$.
     If $x\in T$, then if $V_1\in\mathscr{V}$ is small enough we have
     $V_1\subset  V_0$ and $\gamma_{V_1}^{-1}x\in T$.
     Since $\gamma_{V_1}T=\gamma_{V_0}T$, we see $x\in\gamma_{V_0}T$.
    Conversely, if $x\in\gamma_{V_0}T$, then if  $V_1\in\mathscr{V}$ is small enough
    $\gamma_{V_1}x\in\gamma_{V_0}T=\gamma_{V_1}T$, and so $x\in T$.
    We have shown $T=\gamma_{V_0}T$.
\end{proof}

\begin{cor}\label{cor_patch_hausdorff}
     The local matching topology on $\Patch(X)$ is Hausdorff.
\end{cor}
\begin{proof}
      Clear by Proposition\ref{prop_sufficint_condition_hausdorff} and
      Lemma \ref{lem_patch_atomistic_LIP}.
\end{proof}

By a similar argument we obtain the following:

\begin{lem}
      The abstract pattern space $\Patch_L(X)$ (Example \ref{ex_patch_Gamma_abstract pattern_sp}),
      where $L$ is a non-empty set, is atomistic and has limit inclusion property.
\end{lem}
\begin{cor}
      The local matching topology on $\Patch_L(X)$ is Hausdorff.
\end{cor}



Finally we directly prove on the following product abstract pattern space, the
local matching topology is Hausdorff.

\begin{lem}\label{lem_prod_CL(X)_hausdorff)}
       On the product $\prod_{i\in I}\calC(X)$, where $I$ is a non-empty set,
      $\calC(X)$ is given in Example
      \ref{2X_as_Gamma_abstract pattern_sp}, and the structure of $\Gamma$-abstract
 pattern space is given by Lemma \ref{lemmma_product_abstract pattern_space} and
   Lemma \ref{lem_product_Gamma-abstract pattern_space}, the local matching topology is
     Hausdorff. 
\end{lem}

\begin{proof}
      Take two elements $D=(D_i)_{i\in I}$ and $E=(E_i)_{i\in I}$ from
      $\prod_{i\in I}\calC(X)$ and assume $(D,E)\in\calU_{K,V}$ for any
      $K\in\Cpt(X)$ and $V\in\mathscr{V}$. We will show $D=E$.

      For each $i\in I$,  $x\in D_i$ and $V\in\mathscr{V}$, there is $\gamma_V\in V$
      such that $D\sci\{x\}=(\gamma_V E)\sci \{x\}$. This implies
      $\{x\}=D_i\cap \{x\}=(\gamma_VE_i)\cap\{x\}$, and so $\gamma^{-1}_V x\in E_i$.
      Since the action $\Gamma\curvearrowright X$ is continuous and $E_i$ is closed,
      we see $x\in E_i$ and $D_i\subset E_i$. By symmetry $D_i=E_i$.
      Since $i$ is arbitrary, we have $D=E$.
\end{proof}

\begin{cor}
     The local matching topologies on $\calC(X), \LF(X), \UD_r(X)$, $\UD(X)$ and
     $\UD^I(X)$
 are Hausdorff. 
\end{cor}
\begin{proof}
      Clear by
     Lemma \ref{lem_prod_CL(X)_hausdorff)} since these spaces are included in $\calC(X)$
     or $\prod_{i\in I}\calC(X)$.
\end{proof}




\subsection{The completeness of local matching topology, FLC and the compactness of the
  continuous hull}
\label{Subsection_complete}
Next we prove that under a mild condition the local matching uniform structure
on a subshift is complete.

In this subsection $\Pi$ is a glueable abstract pattern space over $(X,\Gamma)$.

\begin{lem}\label{lem_for_Sigma_complete}
        For each $n=1,2,\ldots$ take $\gamma_n\in\Gamma$ such that
         $\rho_{\Gamma}(e,\gamma_n)<\frac{1}{2^{n+1}}$. Then the following hold:
         \begin{enumerate}
	  \item $\rho_{\Gamma}(\gamma_n\gamma_{n-1}\cdots\gamma_m,e)<\frac{1}{2^m}$ for
               each $n\geqq m\geqq 1$.
          \item For any $m\geqq 1$ 
          the sequence $(\gamma_n\gamma_{n-1}\cdots\gamma_m)_{n\geqq m}$ is a Cauchy 
          sequence.
	 \end{enumerate}
\end{lem}
\begin{proof}
         1. We have 
           \begin{align*}
	         \rho_{\Gamma}(\gamma_n\cdots\gamma_m,e)&\leqq
        \sum_{k=m}^{n-1}\rho_{\Gamma}(\gamma_n\cdots\gamma_k,\gamma_n\cdots\gamma_{k+1})
	       +\rho_{\Gamma}(\gamma_n,e)\\
        &=\sum_{k=m}^n\rho_{\Gamma}(e,\gamma_{k})\\
	    &<  \sum\frac{1}{2^{k+1}}\\
            & <\frac{1}{2^{m}}.
	   \end{align*}
       2. For any $\e>0$, there is $\delta>0$ such that if
        $\gamma,\eta,\zeta\in B(e,1)$ and $\rho_{\Gamma}(\gamma,\eta)<\delta$, then
       $\rho_{\Gamma}(\gamma\zeta,\eta\zeta)<\e$. This follows from the fact that
       $B(e,1)$ is compact and so the multiplication 
      $B(e,1)\times B(e,1)\ni(\gamma,\eta)\mapsto\gamma\eta\in \Gamma$ is uniformly 
      continuous.
      If $n> k\geqq m$ and $k$ is large enough, by 1.,
     \begin{align*}
           \rho_{\Gamma}(\gamma_n\cdots\gamma_{k+1},e)<\delta.
     \end{align*}
     By the definition of $\delta$, we have
     \begin{align*}
           \rho_{\Gamma}(\gamma_n\cdots\gamma_m,\gamma_k\cdots\gamma_m)<\e.
     \end{align*}
     Since $\e$ was arbitrary, we see the sequence is Cauchy.
\end{proof}

\begin{thm}\label{thm_sigma_complete_general_ver}
             Let $\Sigma$ be a supremum-closed subshift of $\Pi$
          on which the local matching topology is Hausdorff.
          Then the local matching uniform structure on $\Sigma$ is complete.
\end{thm}
\begin{proof}
       By \cite[IX, \S 2]{MR1726872}, the local matching uniform structure on $\Sigma$
      is metrizable (see also Corollary \ref{102557_9Aug18}.)
       It suffices to show that any Cauchy sequences in $\Sigma$ converge.
     
      Let $(\calP_n)_n$ be a Cauchy sequence in $\Sigma$.
      Set $K_n=B(x_0,n)$ and $V_n=B(e,\frac{1}{2^{n+1}})\subset\Gamma$
      for each $n=1,2,\ldots$.
      Since it suffices to show a subsequence of $(\calP_n)$ converges, we may assume
 that $(\calP_k,\calP_l)\in\calU_{K_n,V_n}$ for any $n>0$ and $k,l\geqq n$.
      For each $n>0$ there is $\gamma_n\in V_n$ such that
      \begin{align*}
            (\gamma_n\calP_n)\sci K_n=\calP_{n+1}\sci K_n.
      \end{align*}
       By Lemma \ref{lem_for_Sigma_complete}, since $\Gamma$ is complete,
      there is a limit
      \begin{align*}
            \xi_n=\lim_{m\rightarrow\infty}\gamma_m\gamma_{m-1}\cdots\gamma_n
       \in B(e,\frac{1}{2^n})
      \end{align*}
       for each $n>0$. Note that $\xi_n=\xi_{n+1}\gamma_n$ for each $n$.

       Since the group action is continuous, we can take $n_0\in\mathbb{N}$ such that
     if $\gamma\in B(e,\frac{1}{2^{n_0}})$, then $\gamma x_0\in B(x_0,1)$.
       If $n,m\geqq n_0$ and $n<m$, then since
       \begin{align*}
       	    \xi_{m+1}K_m=B(\xi_{m+1}x_0,m)\supset B(x_0,m-1)\supset B(x_0,n)=K_n,
       \end{align*}
       we have
       \begin{align*}
       	   (\xi_m\calP_m)\sci K_n&=(\xi_{m+1}((\gamma_m\calP_m)\sci K_m))\sci K_n\\
                               &=(\xi_{m+1}(\calP_{m+1}\sci K_m))\sci K_n\\
                               &=(\xi_{m+1}\calP_{m+1})\sci K_n.
       \end{align*}
       By induction we have
         \begin{align}
	      (\xi_m\calP_m)\sci K_n=(\xi_{n+1}\calP_{n+1})\sci K_n
              \label{eq1_Sigma_complete}
	 \end{align}
        for each  $n,m>n_0$ with $m>n$.
         This means that
          \begin{align}
	      (\xi_{n+1}\calP_{n+1})\sci K_n\leqq(\xi_{n+2}\calP_{n+2})\sci K_{n+1}
               \label{eq2_Sigma_complete}
	  \end{align}
         for any $n>n_0$.

       Set 
       \begin{align*}
	     \calQ_k=\bigvee\{(\xi_{n+1}\calP_{n+1})\sci K_n\mid n>k\}
       \end{align*}
        for each $k>n_0$.
       We need to show that such a supremum exists.
       To this objective it suffices to show that
       $\Xi_k=\{(\xi_{n+1}\calP_{n+1})\sci K_n\mid n>k\}$ is locally finite and 
       pairwise compatible.
       By (\ref{eq1_Sigma_complete}), we have
        \begin{align*}
	    (\xi_{m+1}\calP_{m+1})\sci K_m\sci K_n=
               (\xi_{n+1}\calP_{n+1})\sci K_n
	\end{align*}
        for any $n,m$ with $k<n<m$, and so $\Xi_k$ is pairwise compatible.
       To prove $\Xi$ is locally finite, take a closed ball $B$.
       For any sufficiently large $n$, we have $K_n\supset B$, and so if
       $m$ is larger than this $n$ we have by (\ref{eq1_Sigma_complete})
       \begin{align*}
       	    (\xi_{m+1}\calP_{m+1})\sci K_m\sci B=(\xi_{m+1}\calP_{m+1})\sci K_n\sci B
              =(\xi_{n+1}\calP_{n+1})\sci K_n\sci B,
       \end{align*}
        and so $\Xi\sci B$ is finite. Since $B$ was arbitrary, $\Xi$ is locally 
         finite. Thus $\calQ_k$ is well-defined and is in $\Sigma$ since 
          $\Sigma$ is supremum-closed.

          By $\Xi_1\supset\Xi_k$, we have $\calQ_1\geqq\calQ_k$ for each $k$.
          On the other hand, by (\ref{eq2_Sigma_complete}) 
           $\calQ_k\geqq(\xi_{n+1}\calP_{n+1})\sci K_n$ for any $n$ and so
          $\calQ_k\geqq\calQ_1$; we have shown $\calQ_1=\calQ_k$ for any $k>0$.

         Finally $\calQ_1$ is the limit of $(\calP_n)$, since for each $k>n_0$,
         (\ref{eq1_Sigma_complete}) implies that
          \begin{align*}
	         \calQ_1\sci K_k&=\bigvee\{\xi_{n+1}\calP_{n+1}\sci K_k\mid n>k\}\\
                               &=(\xi_{k+1}\calP_{k+1})\sci K_k
	  \end{align*}
          and so $\calP_{k+1}\in\calU_{K_k,V_k}(\calQ_1)$.
          (Note that $\calU_{K_{k+1},V_{k+1}}\subset\calU_{K_k,V_k}$.)
\end{proof}

\begin{cor}
      On $C_b(X,Y)$ (Lemma \ref{lem_Cb_atomistic_LIP}), $\Patch(X)$, $\Patch_L(X)$
        (Example \ref{ex_patch_Gamma_abstract pattern_sp})
      $\UD^I_r(X)$ (Example \ref{ex_multi_set}, Example \ref{2X_as_Gamma_abstract pattern_sp}),
 $\calC(X), \LF(X),\UD(X)$
       (Example \ref{example_UD}, Example \ref{2X_as_Gamma_abstract pattern_sp}),
       and $\WDC_r(X)$ (Example \ref{ex_Gamma-abstract pattern_space_measures}),
       the local matching uniform structures are complete.
\end{cor}

\begin{rem}
        This theorem is similar to
       Proposition 2.1 in \cite{MR1798991}
      and Theorem 3.10 in \cite{MR2137108}.
      The former proves the completeness of discrete and closed subsets of
      $\sigma$-compact locally compact abelian group, which is included in
      the latter.
      The latter proves the completeness of the space of discrete subsets of
      a $\sigma$-compact space, which is more general than Theorem
      \ref{thm_sigma_complete_general_ver} in that the author does not assume
      the existence
      of metrics in the space and the group, but less general in that
       it only deals with discrete sets, rather than any abstract patterns.
\end{rem}

Finally we define FLC  and the continuous hull
for abstract patterns and prove that FLC implies the compactness
of the continuous hull, if the action $\Gamma\curvearrowright X$ is proper.

\begin{defi}
      Take $\calP\in\Pi$. The continuous hull $X_{\calP}$ of $\calP$ is defined by
      \begin{align*}
            X_{\calP}=\overline{\{\gamma\calP\mid\gamma\in\Gamma\}},
      \end{align*}
       where the closure is taken with respect to the local matching topology.
\end{defi}

\begin{defi}
       An abstract pattern $\calP\in\Pi$ is said to have finite local complexity (FLC)
      if whenever we take a compact $K\subset X$, the set
       \begin{align*}
	    \{(\gamma\calP)\sci K\mid\gamma\in\Gamma\}
       \end{align*}
       is finite modulo the group action  $\Gamma\curvearrowright\Pi$.
\end{defi}

\emph{In what follows, 
$\Sigma$ is a supremum-closed, atomistic subshift of $\Pi$ with limit inclusion
property.}

We will use the following lemma to prove the compactness of a continuous hull.
\begin{lem}\label{lem_for_FLC_cpt}
      Assume the action $\Gamma\curvearrowright X$ is proper.
     Take $\calP_1,\calP_2\ldots$ from $\Sigma$.
     Assume
$\{\calP_n\sci K\mid n=1,2,\ldots\}$ is finite modulo $\Gamma\curvearrowright\Pi$, for any $K\in\Cpt(X)$.
       Then for any compact $K\subset X$ and $V\in\mathscr{V}$, there is a subsequence
       $(\calP_{n_j})_{j=1,2,\cdots}$ of $(\calP_n)_n$  such that
       \begin{align*}
	   (\calP_{n_j},\calP_{n_k})\in\calU_{K,V}
       \end{align*}
         for any $j$ and $k$.
\end{lem}
\begin{proof}
       Take $K\in\Cpt(X)$ and $V\in\mathscr{V}$, and we will prove we can take a
      subsequence of $(\calP_n)_n$ with the above property.

 Take $R>0$ large enough so that $B(x_0,R)\supset K$
      hold.

 Set $K'=B(x_0,R+1)$. By the second condition in the statement of this lemma,
        \begin{align*}
	     \{\calP_n\sci K'\mid n=1,2,\cdots\}
	\end{align*}
        is finite modulo $\Gamma$-action.
       We can take an increasing map $\sigma\colon\mathbb{N}\rightarrow\mathbb{N}$
       and elements $\eta_1,\eta_2,\cdots\in\Gamma$ such that
       \begin{align}
	   \calP_{\sigma(1)}\sci K'=\eta_n(\calP_{\sigma(n)}\sci K')\label{144025_8Aug18}
       \end{align}
        for each $n=1,2,\ldots$.

        We may assume $\calP_{\sigma(1)}\sci K'\neq 0$ ($0$ is the zero element in $\Pi$,
        which is unique by Lemma \ref{uniqueness_zero_element}), since if it is $0$,
        $\calP_{\sigma(n)}\sci K'=0=\calP_{\sigma(m)}\sci K'$
        for each $n$ and $m$, which means
         $(\calP_{\sigma(n)},\calP_{\sigma(m)})\in\calU_{K,V}$.
         In the case where it is not $0$,
        the set $\supp\calP_{\sigma(n)}\sci K'$ is non-empty and included in $K'$ by
         Lemma \ref{lemma_support_calP_sci_C}.
         Since the action $\Gamma\curvearrowright X$ is proper, the set
         $\{\eta_1,\eta_2,\ldots\}$ is relatively compact.
          We can take an increasing map $\tau\colon\mathbb{N}\rightarrow\mathbb{N}$
         such that
         \begin{itemize}
	  \item $\rho_X(\eta_{\tau(n)}^{-1}\eta_{\tau(m)}x_0,x_0)<1$, and
          \item $\eta_{\tau(n)}^{-1}\eta_{\tau(m)}\in V$
	 \end{itemize}
         for each $n,m\in\mathbb{N}$.

          If $n$ and $m$ are natural numbers, by \eqref{144025_8Aug18} we have
          \begin{align*}
	        \eta_{\tau(n)}(\calP_{\sigma\circ\tau(n)}\sci K')&=
	                \calP_{\sigma(1)}\sci K'
	                =\eta_{\tau(m)}(\calP_{\sigma\circ\tau(m)}\sci K').
	  \end{align*}
          By multiplying both sides by $\eta_{\tau(m)}^{-1}$ and using
          the fact that the cutting-off
      operation is equivariant
       (Definition \ref{def_gamma-abstract pattern_space}), we obtain
         \begin{align}
	      (\eta_{\tau(m)}^{-1}\eta_{\tau(n)}\calP_{\sigma\circ\tau(n)})\sci
	            (\eta_{\tau(m)}^{-1}\eta_{\tau(n)}K')=
	              \calP_{\sigma\circ\tau(m)}\sci K'.\label{144209_8Aug18}
	 \end{align}
          Since $\eta_{\tau(m)}^{-1}\eta_{\tau(n)}K'=
            B(\eta_{\tau(m)}^{-1}\eta_{\tau(n)}x_0,R+1)\supset B(x_0,R)\supset K$ by
           the definition of $\tau$, we have, by cutting off both sides
           of \eqref{144209_8Aug18} by $K$,
           \begin{align*}
	          (\eta_{\tau(m)}^{-1}\eta_{\tau(n)}\calP_{\sigma\circ\tau(n)})\sci K
	          =\calP_{\sigma\circ\tau(m)}\sci K,
	   \end{align*}
           which means $(\calP_{\sigma\circ\tau(m)},\calP_{\sigma\circ\tau(n)})\in
           \calU_{K,V}$ by the definition of $\tau$.
\end{proof}

The following diagonalization argument is well-known, compare for example \cite{MR2446623}.

\begin{thm}\label{thm_compact_continuous_hull}
      Assume the action $\Gamma\curvearrowright X$ is proper. Take $\calP\in\Pi$ and
      assume it has FLC.
       Then the continuous hull $X_{\calP}$ is compact.
\end{thm}

\begin{proof}
       Since on $\Sigma$ the local matching topology is complete (Theorem \ref{thm_sigma_complete_general_ver}),
        is suffices to show that $\{\gamma\calP\mid\gamma\in\Gamma\}$ is totally bounded.
       To this aim we take a sequence $(\gamma_n\calP)$ from this set and prove there is
      a Cauchy subsequence.

       Take a sequence of compact sets $K_1,K_2\ldots\in\Cpt(X)$ and a sequence
        $V_1,V_2,\ldots\mathscr{V}$ such that $\{\calU_{K_n,V_n}\mid n=1,2,\ldots\}$ is
        a fundamental system of entourages and
         $\calU_{K_n,V_n}^2\subset\calU_{K_{n-1},V_{n-1}}$ for each $n=2,3,\ldots$.
         For example set $K_n=B(x_0,R_n)$ and $V_n=B(e,r_n)$, where $(R_n)$ is rapidly
         increasing and $(r_n)$ is rapidly decreasing.
         Then we have, for each $n,m$ with $n<m$,
          \begin{align}
	       \calU_{K_m,V_m}\calU_{K_{m-1},V_{m-1}}\cdots\calU_{K_n,V_n}\subset
	         \calU_{K_{n-1},V_{n-1}}.\label{151515_8Aug18}
	  \end{align}

          By Lemma \ref{lem_for_FLC_cpt}, we can take a subsequence $(\calP^{(1)}_n)_{n=1,2,\ldots}$ of $(\gamma_n\calP)$ such that
           $(\calP^{(1)}_n,\calP^{(1)}_m)\in\calU_{K_1,V_1}$ for any $n$ and $m$.
          We can further take a subsequence $(\calP^{(2)}_n)_n$ of $(\calP^{(1)}_n)$ such
         that $(\calP_n^{(2)},\calP_m^{(2)})\in\calU_{K_2,V_2}$ for each $n,m$.
          We proceed in this way to obtain a sequences $(\calP^{(k)}_n)_n$, $k=1,2,\ldots$.
           Set $\calQ_n=\calP^{(n)}_n$. Then by \eqref{151515_8Aug18}, the sequence
       $(\calQ_n)_n$ is
 a Cauchy subsequence of $(\gamma_n\calP)_n$.
\end{proof}

     FLC of an abstract pattern is not necessarily inherited by local derivability,
     which was defined in \cite{MR1132337} and \cite{Nagai3rd}. For example, $\mathbb{Z}$ has FLC, but
      the function $\mathbb{R}\ni t\mapsto \sin(2\pi t)\in\mathbb{R}$ does not have
      FLC, although the latter is locally derivable from the former.
     However, the fact that the continuous hull is compact is often inherited, since
    if $\calQ$ is locally derivable from $\calP$ and
      $\calQ$ lies in a set $\Sigma$ on which the
 local matching topology is complete,
 the map
     \begin{align*}
          \{\gamma\calP\mid\gamma\in\Gamma\}\ni\gamma\calP\mapsto\gamma\calQ
          \in\{\gamma\calQ\mid\gamma\in\Gamma\}
     \end{align*}
      is uniformly continuous and extended to
       $X_{\calP}\rightarrow X_{\calQ}$.

       With an additional assumption, we have the following:

       \begin{thm}
	   Take $\calP\in\Sigma$  and assume  the set
	   $A(\calP\sci K)$ is finite for each $K\in\Cpt(X)$.
	If $X_{\calP}$
	   is compact, then the abstract pattern $\calP$ has FLC.
       \end{thm}
        \begin{proof}
	    Assume for some $K\in\Cpt(X)$ the set
	     \begin{align*}
	          \{(\gamma\calP)\sci K\mid\gamma\in\Gamma\}
	     \end{align*}
	     is infinite up to $\Gamma$-action. Then there are
	      $\gamma_1,\gamma_2,\ldots\in\Gamma$ such that if $n\neq m$, 
	       $(\gamma_n\calP)\sci K$ and $(\gamma_m\calP)\sci K$ are not equivalent
	       with respect to $\Gamma$-action. Since $X_{\calP}$ is compact, by
	 passing to a
	 subsequence if necessary, we may assume there is $\calQ\in X_{\calP}$
	  such that $\gamma_n\calP\rightarrow\calQ$ as $n\rightarrow\infty$.

	     Take a compact neighborhood $V$ of $e\in\Gamma$ and a compact $K'\subset X$
	 such that if $\xi,\zeta\in V$ and $x\in K$, then $\xi^{-1}\zeta x\in K'$.
	 By passing to a subsequence if necessary, we may assume for each $n$ there is
	   $\eta_n\in V$ such that
	    \begin{align*}
	          (\eta_n\gamma_n\calP)\sci K'=\calQ\sci K'.
	    \end{align*}

	 If $\calR\in A((\gamma_n\calP)\sci K)$, then
	   $\eta_n\calR\in A((\eta_n\gamma_n\calP)\sci K')\subset A(\eta_1\gamma_1\calP)$.
	 We have $\eta_n\calR\in A((\eta_1\gamma_1\calP)\sci K')$.
	   Since the latter is finite, there are distinct $n$ and $m$ such that
	 \begin{align*}
	    \eta_n A((\gamma_n\calP)\sci K)=\eta_m
	    A((\gamma_m\calP)\sci K),
	 \end{align*}
	 which implies
	 \begin{align*}
	      \eta_n((\gamma_n\calP)\sci K)=\eta_m((\gamma_m\calP)\sci K).
	 \end{align*}
	  This contradicts the assumption at the beginning and we see
	     $\calP$ has FLC.
	\end{proof}

	This means that for Delone sets and tilings with finitely many tile types
	up to $\Gamma$-action, FLC is equivalent to compactness of the continuous hull,
	which is well-known for the case of $X=\Rd$.

\bibliographystyle{amsplain}
\bibliography{tiling}

\end{document}